\documentclass{amsart}
\usepackage[utf8]{inputenc}
\usepackage{amsfonts}
\usepackage{amscd}
\usepackage{verbatim}
\usepackage{hyperref}
\usepackage[all]{xy}
\usepackage{pictexwd, amsthm, amssymb, amsmath, amsxtra, graphicx, psfrag, latexsym}
\usepackage{subfigure}

\usepackage{enumerate}        %

\newenvironment{proof*}{\noindent\emph{Proof}}{$\square$\smallskip}

\newtheorem{theorem}{Theorem}[section]

\newtheorem{Definition}[theorem]{Definition}
\newtheorem{lemma}[theorem]{Lemma}
\newtheorem{Example}[theorem]{Example}

\newtheorem{Remark}[theorem]{Remark}

\newtheorem{proposition}[theorem]{Proposition}

\newtheorem{Exercise}[theorem]{Exercise}
\newtheorem{Exercises}[theorem]{Exercises}
\newtheorem{Notation}[theorem]{Notation}
\newtheorem{Convention}[theorem]{Convention}

\newenvironment{definition}{\begin{Definition}\normalfont}{\end{Definition}}
\newenvironment{example}{\begin{Example}\normalfont}{\end{Example}}
\newenvironment{remark}{\begin{Remark}\normalfont}{\end{Remark}}

\newcommand{\lkb}[2]{\ensuremath{\mathrm{lk}_{\downarrow}^{QV}(x^{#1}a^{#2})}} 

\title[QV, QT, QF Finiteness Properties]{Quasi-Automorphism Groups of Type $F_\infty$}
\author[S. Audino]{Samuel Audino}
\email{saudino480@gmail.com} 

\author[D. R. Aydel]{Delaney R. Aydel}
\address{Department of Mathematics\\
Miami University\\
Oxford, OH 45056 U.S.A}
\email{aydeld@miamioh.edu}

\author[D. S. Farley]{Daniel S. Farley}
\address{Department of Mathematics\\
Miami University\\
Oxford, OH 45056 U.S.A}
\email{farleyds@miamioh.edu}

\begin{document}

\begin{abstract}
The groups $QF$, $QT$, $\bar{Q}T$, $\bar{Q}V$, and $QV$ are groups of quasi-automorphisms of the infinite binary tree. Their names indicate a similarity with Thompson's well-known groups $F$, $T$, and $V$. 

We will use the theory of diagram groups over semigroup presentations to prove that all of the above groups (and several generalizations) have type $F_{\infty}$. Our proof uses certain types of hybrid diagrams, which have properties in common with both planar diagrams and braided diagrams. The diagram groups defined by hybrid diagrams also act properly and isometrically on CAT($0$) cubical complexes.
\end{abstract}



\maketitle

\section{Introduction}

Let $\Gamma$ be a graph. A \emph{quasi-automorphism} of $\Gamma$ is a bijection of the vertices that preserves adjacency, with at most finitely many exceptions. Following the notation of \cite{NJG}, we will let $QV$ denote the group of quasi-automorphisms $h$ of the infinite binary tree that also take the left and right children of a given vertex $v$ to the left and right children of $h(v)$, again with at most finitely many exceptions. The notation ``$QV$" indicates that $QV$ is a collection of quasi-automorphisms that bears a family resemblance to Thompson's group $V$. (A standard reference for Thompson's groups is \cite{CFP}. We will assume a basic familiarity with that source or its equivalent throughout this article.)

Groups of quasi-automorphisms have been the subject of several recent studies. Lehnert conjectured in his thesis that the group $QV$ is a universal group with context-free coword problem; i.e., a universal $coCF$ group. Bleak, Matucci, and Neunh\"{o}ffer \cite{Bleak} have produced an embedding of $QV$ into Thompson's group $V$, and thus proved that Lehnert's conjecture is equivalent to the conjecture that $V$ is itself a universal $coCF$ group. More recently, Nucinkis and St. John-Green \cite{NJG} have studied the finiteness properties of $QV$ and related groups. They introduced additional groups $QF$, $QT$, $\bar{Q}T$, and $\bar{Q}V$. The groups $QF$ and $QT$ are natural subgroups of $QV$ that preserve (respectively) the linear and cyclic orderings of the ends of the infinite binary tree (and therefore bear a family resemblance to Thompson's groups $F$ and $T$, respectively). The groups $\bar{Q}T$ and $\bar{Q}V$ are analogous groups that act as quasi-automorphisms on the union of an infinite binary tree with an isolated point. Nucinkis and St. John-Green show that the groups 
$QF$, $\bar{Q}T$, and $\bar{Q}V$ have type $F_{\infty}$, and also compute explicit finite presentations for these groups. Whether $QT$ and $QV$ are finitely presented (and thus, more particularly, of type $F_{\infty}$) is left as an open problem in \cite{NJG}.  

The third author showed (in \cite{FH2}) that $QV$ is a braided diagram group over a semigroup presentation. This description suggests an approach to proving the $F_{\infty}$ property for $QV$. Since \cite{Far(fp)} shows that a class of braided diagram groups (including Thompson's group $V$) have type $F_{\infty}$, and in fact other classes of diagram groups were shown to have type $F_{\infty}$ in \cite{Far(th)} and \cite{Far(fp)}, it is at least plausible that some approach inspired by the theory of diagram groups could establish the $F_{\infty}$ property for $QV$ and $QT$. (We note that the original proofs that Thompson's groups $F$, $T$, and $V$ have type $F_{\infty}$ were given by Brown \cite{Brown} and by Brown and Geoghegan \cite{BG} in the 1980s.)

Nucinkis and St. John-Green show, however, that the hypotheses of the main theorem in \cite{Far(fp)} are satisfied by neither $QT$ nor $QV$. In fact, as also noted in \cite{NJG}, even the much more general main theorem of \cite{Thumann} does not apply to either of $QT$ and $QV$.

The goal of the present article is to extend the diagram-group methods of \cite{Far(th)} and \cite{Far(fp)} to the groups $QF$, $QT$, $QV$, $\bar{Q}T$, and $\bar{Q}V$. We will show that all of these groups can be described using the theory of diagram groups over semigroup presentations. Indeed, all of these groups are diagram groups over the same semigroup presentation, namely
$\mathcal{P} = \langle x, a \mid x = xax \rangle$, although the specific types of diagram vary from group to group. Three types of diagram groups have been considered in the literature: planar, annular, and braided diagram groups. All were introduced by Guba and Sapir in \cite{GS}, which devotes by far the greatest attention to planar diagram groups (which are usually simply called diagram groups). The papers \cite{Far(pg)} and \cite{Far(fp)} consider the annular and braided diagram groups in more detail. Here we will introduce hybrid diagram groups that combine properties of multiple diagram group types. For instance, the group $QF$ is a special type of diagram group over $\mathcal{P}$, in which the diagrams exhibit both planar and braided behavior at the same time. We will use such hybrid diagrams to prove that the groups $QF$, $QT$, $QV$, $\bar{Q}T$, and $\bar{Q}V$ all act properly by isometries on CAT($0$) cubical complexes, and that all have type $F_{\infty}$. In fact, our methods extend with equal ease to the case of an arbitrary finite number of binary trees and isolated vertices (see Section \ref{section:generalization}), and the  case of $n$-ary trees (for fixed $n \geq 2$) is different only in the details. It even seems likely that our argument generalizes to other, non-regular, trees, although this is more speculative, and we attempt no general statement about such cases here.

We note that it is probably possible to extend the main theorem of \cite{Thumann} to prove the $F_{\infty}$ property in the cases under consideration here. This is only a guess, however, since the authors claim little familiarity with the methods of \cite{Thumann}.

We will now offer an outline of the argument. In Section \ref{section:bdgactions}, we will give a rapid introduction to the basic theory of diagram groups (including all three types: planar, annular, and braided), and describe the natural cubical complexes on which such groups act, including a description of the links of vertices. In Section \ref{section:QFQTQV}, we will give careful definitions of the groups $QF$, $QT$, and $QV$, and describe how to represent elements of each group as ``hybrid" diagrams. Section \ref{section:QFQTQV} also includes a description of natural complexes on which $QF$ and $QT$ act; these arise as convex (and thus CAT($0$) by \cite{CW}) subcomplexes of $QV$.  (We will in fact confine our attention to $QF$, $QT$, and $QV$ alone, without sketching a general theory of ``hybrid" diagrams. Nevertheless, we hope that the ideas indicated in Section \ref{section:QFQTQV} may be of some independent interest.) Section \ref{section:Finfinity} shows that $QF$, $QT$, and $QV$ have type $F_{\infty}$. Our argument follows the long-established method given by Ken Brown in \cite{Brown}. 
Section \ref{section:generalization} sketches some possible further developments, including sketches of the proofs that $\bar{Q}T$ and $\bar{Q}V$ have type $F_{\infty}$ (as already proved by \cite{NJG}), and the other generalizations briefly described above.

\section{Braided diagram groups and actions on associated complexes} \label{section:bdgactions}

\subsection{Definition of braided diagram groups}

To define braided diagram groups, we must first define braided diagrams over semigroup presentations. 

\begin{definition}  \label{definition:diagram} (Braided diagrams over a semigroup presentation) 
Let $\mathcal{P}$ be a \emph{semigroup presentation}; thus $\mathcal{P} = \langle \Sigma \mid
\mathcal{R} \rangle$, where $\Sigma$ is a set (to be regarded as an alphabet) and 
$\mathcal{R} \subseteq \Sigma^{+} \times \Sigma^{+}$, where $\Sigma^{+}$ is the set of all non-empty positive words in the symbols $\Sigma$. We view the elements of $\mathcal{R}$ as equalities 
between elements of $\Sigma^{+}$ (i.e., as \emph{relations}). For technical reasons, we impose the additional restriction that $(u,u) \notin \mathcal{R}$, for any $u \in \Sigma^{+}$.

A \emph{braided diagram $\Delta$ over $\mathcal{P}$} is a labelled ordered topological space formed by making identifications among three types of components: wires, transistors, and the frame.
\begin{itemize}
\item A \emph{wire} is a homeomorphic copy of $[0,1]$. The ``$0$" end is the bottom of the wire, and the ``$1$" end is the top.  
\item A \emph{transistor} is a homeomorphic copy of $[0,1]^{2}$. Each transistor has well-defined top, bottom, left, and right sides (in the obvious senses: the top is $[0,1] \times \{ 1 \}$, the bottom
is $[0,1] \times \{ 0 \}$, etc.). These sides are part of the transistor's definition. The top and bottom sides have equally obvious left-to-right orderings. (We make no use of any ordering of the sides.)
\item The \emph{frame} is a homeomorphic copy of $\partial [0,1]^{2}$. It has well-defined top, bottom, left, and right sides, just as a transistor does. The top and bottom sides have obvious left-to-right orderings.
\end{itemize}
To form a braided diagram $\Delta$ over $\mathcal{P}$, we begin with a finite non-empty collection
$W(\Delta)$ of wires, a labelling function $\ell: W(\Delta) \rightarrow \Sigma$, a finite (possibly empty) collection $T(\Delta)$ of transistors, and a frame. Each endpoint of each wire is then attached either to a transistor or to the frame. The bottom of a wire is attached either to the top of a transistor or to the bottom of the frame; the top of a wire is attached either to the bottom of a transistor or to the top of the frame. Moreover, the images of any two wires in the quotient must be disjoint.

The resulting labelled oriented quotient space is called a \emph{braided diagram over $\mathcal{P}$} if the following two conditions are also satisfied:

\begin{enumerate}
\item Note that each transistor $T$ inherits a top and bottom label from the labels of the wires it touches. (The points at which the wires meet transistors are called \emph{contacts}.) These labels are words in $\Sigma^{+}$, obtained by reading the labels of connecting wires from left to right.

Let $\ell_{top}(T)$ and $\ell_{bot}(T)$ denote the top and bottom labels, respectively. We require that $(\ell_{top}(T), \ell_{bot}(T)) \in \mathcal{R}$ or that $(\ell_{bot}(T), \ell_{top}(T)) \in \mathcal{R}$, for each transistor $T$.

\item For transistors $T_{1}$, $T_{2}$ of $\Delta$, write $T_{1} \preccurlyeq T_{2}$ if there is a wire whose bottom contact is a point on the top of $T_{1}$ and whose top contact is a point on the bottom of $T_{2}$. Let $<$ be the transitive closure of $\preccurlyeq$. We require that $<$ be a strict partial order on the transistors of $\Delta$.  

(Equivalently: suppose that the braided diagram $\Delta$ is drawn in the plane, such that each transistor is enclosed by the frame, and the sides of the transistors and frame are parallel to the coordinate axes. We require that it be possible to arrange the transistors and the frame in this fashion in such a way that each wire can be embedded monotonically; i.e., so that the $y$-coordinate in the embedded image increases as we move from the bottom of the wire to the top.)
\end{enumerate}
\end{definition}

\begin{definition} (Planar and annular diagrams) Let $\Delta$ be a braided diagram over the semigroup presentation $\mathcal{P}$. If $\Delta$ admits an embedding $h: \Delta \rightarrow \mathbb{R}^{2}$ into the plane that preserves the left-right and top-bottom orientations on the transistors and frame, then we say that $\Delta$ is \emph{planar}.

We say that $\Delta$ is \emph{annular} if it can be similarly embedded in an annulus. Or, more precisely, suppose that we replace the frame $\partial( [0,1]^{2})$ with a pair of disjoint circles, each of which is given the standard counterclockwise orientation, in place of the usual left-right orientations on the top and bottom of $\partial( [0,1])^{2}$. We further give both circles basepoints, which are to be disjoint from all contacts. Transistors and wires are defined as before, and their attaching maps are subject to the same restrictions as before. We say that the resulting diagram is \emph{annular} if the result may be embedded in the plane, again preserving the left-right orientations of the transistors. We think of the ``top" circle as the inner ring of the annulus and the ``bottom" circle as the outer ring.

[Here it may be helpful to view the transistors as having the counterclockwise orientation on their ``top'' and ``bottom'' faces, where the ``top" faces the interior boundary circle of the annulus and the ``bottom" faces the external boundary of the annulus.]
\end{definition}

\begin{definition} (Equivalence of braided diagrams) Two braided diagrams $\Delta_{1}$ and $\Delta_{2}$ are \emph{equivalent} if there is a homeomorphism $h$ between them, such that $h$ preserves the labelings of wires and all orientations (left-right and top-bottom) on all transistors and the frame.
\end{definition}

\begin{definition} \label{definition:uv} ($(u,v)$-diagrams) Let $\Delta$ be a braided semigroup diagram; let $u, v \in \Sigma^{+}$. We can define the \emph{top} and \emph{bottom labels} of $\Delta$ by reading the labels of the wires that connect to the frame, from left to right, just as we defined the top and bottom labels of an individual transistor above. We say that $\Delta$ is a \emph{braided $(u,v)$-diagram}
if the top label of $\Delta$ is $u$ and the bottom label is $v$. 

In some cases, it is not important to specify the bottom label. We say that $\Delta$ is 
a \emph{braided $(u, \ast)$-diagram} if the top label of $\Delta$ is $u$, and the bottom label is arbitrary.
\end{definition}

\begin{definition} \label{definition:concatenation}
(Concatenation) If $\Delta'$ is a braided $(u,v)$-diagram and $\Delta''$ is a braided $(v,w)$-diagram, then the \emph{concatenation} $\Delta' \circ \Delta''$ is defined by stacking the diagrams,
$\Delta'$ on top of $\Delta''$.
\end{definition}

\begin{remark}
We note (for the sake of clarity) that the basepoints on the inner and outer circles of an annular diagram $\Delta$ are needed in Definitions \ref{definition:uv} and \ref{definition:concatenation}. Here, the ``top" label of $\Delta$ is to be read counterclockwise from the top (inner) basepoint, while the bottom label of $\Delta$ is similarly read counterclockwise from the outer circle's basepoint. 

If $\Delta_{1}$ and $\Delta_{2}$ are annular $(u,v)$- and $(v,w)$-diagrams (respectively, for $u$, $v$, $w \in \Sigma^{+}$), then the concatenation $\Delta_{1} \circ \Delta_{2}$ is the result of identifying the outer circle of $\Delta_{1}$ with the inner circle of $\Delta_{2}$ at the chosen basepoints (while also, of course, matching the other contacts in counterclockwise order).
\end{remark}

\begin{figure}[h]
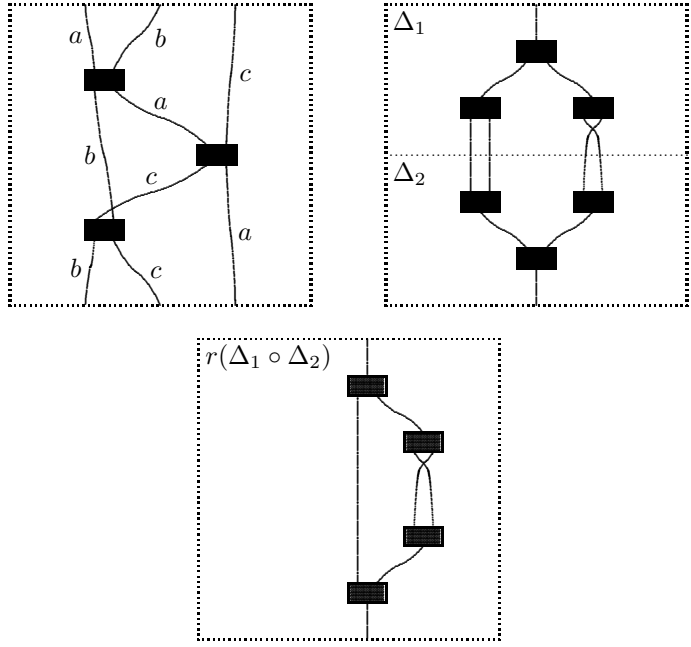
  
\vbox{\hspace{1cm}\beginpicture
\setcoordinatesystem units <.125 cm,.125cm> point at 0 0
\setplotarea x from -2 to 2, y from -2 to 2
\linethickness=.9pt

\setdots <2 pt >
\putrectangle corners at -24 16 and -56 -16

\putrectangle corners at 16 16 and -16 -16
\plot 16 0  -16 0 /

\put{$\Delta_{1}$} at -13.5 14
\put{$\Delta_{2}$} at -13.5 -2

\setsolid
\setshadegrid span <.1pt>
\shaderectangleson

\putrectangle corners at -48 9 and -44 7 
\putrectangle corners at -48 -7 and -44 -9 
\putrectangle corners at -36 1 and -32 -1 

\putrectangle corners at -2 12 and 2 10  
\putrectangle corners at -8 6 and -4 4 
\putrectangle corners at 8 6 and 4 4 

\putrectangle corners at -2 -12 and 2 -10  
\putrectangle corners at -8 -6 and -4 -4 
\putrectangle corners at 8 -6 and 4 -4 

\plot 0 16   0  12 / 
\plot -7 4   -7 -4 / 
\plot -5 4   -5 -4 / 
\plot 0 -16   0 -12 / 

\setquadratic
\plot -48 16  -47.75 13.8125  -47.5 12.5 / 
\plot -47.5 12.5   -47.25 11.1875 -47 9 / 
\put {$a$} at -49 12.5 
\plot -40 16  -41.25 13.8125 -42.5 12.5 / 
\plot -42.5 12.5   -43.75 11.1875 -45 9 / 
\put {$b$} at -40 12.5
\plot -32 16  -32.25 11.3125  -32.5 8.5 / 
\plot -32.5 8.5  -32.75 5.6875 -33 1 / 
\put {$c$} at -31 8.5 
\plot -45 7   -42.5 5.125  -40 4 / 
\plot -40 4   -37.5 2.875  -35 1 / 
\put {$a$} at -40 5.5
\plot -47 7  -46.5 2.625   -46 0 / 
\plot -46 0  -45.5 -2.625  -45 -7 / 
\put {$b$} at -47.5 0 
\plot -35 -1   -38 -2.875  -41 -4 / 
\plot -41 -4   -44 -5.125   -47 -7 / 
\put {$c$} at -41 -2.5
\plot -33 -1   -32.75 -5.6875   -32.5 -8.5 / 
\plot -32.5 -8.5   -32.25 -11.3125   -32 -16 / 
\put {$a$} at -31 -8.5
\plot -47 -9   -47.25 -11.1875   -47.5 -12 / 
\plot -47.5 -12   -47.75 -13.8125   -48 -16 / 
\put {$b$} at -49 -12
\plot -45 -9   -43.75 -11.1875   -42.5 -12.5 /
\plot -42.5 -12.5  -41.25 -13.8125   -40 -16 /  
\put {$c$} at -40.5 -12.5

\plot -1 10     -2.25 8.75    -3.5 8 / 
\plot -3.5 8    -4.75 7.25    -6 6 / 

\plot 1 10     2.25 8.75    3.5 8 / 
\plot 3.5 8    4.75 7.25    6 6 / 

\plot 5 4    5.375 3.375    5.75 3 / 
\plot 5.75 3    6.125 2.625    6.5 2 /
\plot 6.5 2    6.75 .125   7 -4 / 

\plot 7 4   6.625 3.375 6.25 3 /
\plot 6.25 3    5.875 2.625    5.5 2 /
\plot 5.5 2    5.25 .125    5 -4 / 

\plot -1 -10     -2.25 -8.75    -3.5 -8 / 
\plot -3.5 -8    -4.75 -7.25    -6 -6 / 
\plot 1 -10     2.25 -8.75    3.5 -8 / 
\plot 3.5 -8    4.75 -7.25    6 -6 / 

\endpicture}\hfil

\vspace{10pt}

\vbox{\hspace{1cm}\beginpicture
\setcoordinatesystem units <.125 cm,.125cm> point at 0 0
\setplotarea x from -2 to 2, y from -2 to 2
\linethickness=.7pt


\setdots <2 pt >
\putrectangle corners at 16 16 and -16 -16

\put{$r(\Delta_{1} \circ \Delta_{2})$} at -8.5 14

\setsolid
\setshadegrid span <.3pt>
\shaderectangleson

\putrectangle corners at 0 12 and 4 10  

\putrectangle corners at 10 6 and 6 4 

\putrectangle corners at 0 -12 and 4 -10  

\putrectangle corners at 10 -6 and 6 -4 

\plot 2 16   2  12 / 
\plot 1 10 1 -10 / 
\plot 2 -16   2 -12 / 

\setquadratic
\plot 3 10     4.25 8.75    5.5 8 / 
\plot 5.5 8    6.75 7.25    8 6 / 

\plot 7 4    7.375 3.375    7.75 3 / 
\plot 7.75 3    8.125 2.625    8.5 2 /
\plot 8.5 2    8.75 .125   9 -4 / 

\plot 9 4   8.625 3.375 8.25 3 /
\plot 8.25 3    7.875 2.625    7.5 2 /
\plot 7.5 2    7.25 .125    7 -4 / 

\plot 3 -10     4.25 -8.75    5.5 -8 / 
\plot 5.5 -8    6.75 -7.25    8 -6 / 
\endpicture

}\hfil
\caption{Three examples of braided diagrams over semigroup presentations.}
\label{figure:1}
\end{figure}

\begin{example} Figure \ref{figure:1} gives three examples of braided diagrams over semigroup presentations. On the upper left, we have a braided $(abc,bca)$-diagram $\Delta$ over the semigroup presentation $\mathcal{P} = \langle a, b, c, \mid ab = ba, bc=cb, ac=ca \rangle$. More properly speaking, this is the immersed image of such a diagram under a mapping into the plane. All of the defining features of $\Delta$ are illustrated. The frame appears as a dotted box, and the left-to-right orientations of the transistors are the obvious ones. Note that the apparent crossing of wires above the bottom left transistor represents a double point of the projection, since the wires are necessarily disjoint (by Definition \ref{definition:diagram}) in the original diagram $\Delta$. Note also that it is unnecessary to specify whether the ``$c$" wire crosses over the ``$b$" wire or vice versa, since such overcrossing data is not part of the definition of $\Delta$. (In effect, we allow any two wires of a braided diagram  to pass through each other.) This means that the descriptor ``braided" is a misnomer; there is no true braiding.

In the upper right is the concatenation $\Delta_{1} \circ \Delta_{2}$ of $\Delta_{1}$ with $\Delta_{2}$, where both are diagrams over the semigroup presentation $\mathcal{P} = \langle x \mid x = x^{2} \rangle$.  (We omit the labels of the wires, since each such label is ``$x$".) Here $\Delta_{1}$ is a braided $(x,x^{4})$-diagram and $\Delta_{2}$ is a braided $(x^{4},x)$-diagram. 

At the bottom is the result of removing all dipoles from the concatenation $\Delta_{1} \circ \Delta_{2}$. (See Definition \ref{definition:dipoles}.)
\end{example}
 
It is reasonably clear that, for a fixed word $w \in \Sigma^{+}$, the braided $(w,w)$-diagrams over $\mathcal{P}$ form a semigroup under concatenation. We can define inverses using the idea of a dipole.
 
\begin{definition} \label{definition:dipoles} (Dipoles)
Suppose that $T_{1}$ and $T_{2}$, $T_{1} \preccurlyeq T_{2}$, are transistors in a braided semigroup diagram $\Delta$ over 
$\mathcal{P}$. Let $w_{1}$, $\ldots$, $w_{n}$ be a complete list of wires attached at the top of $T_{1}$, listed in the left-to-right order of their attaching contacts. We say that $T_{1}$ and $T_{2}$ form a \emph{dipole} if:
\begin{enumerate}
\item the tops of the wires $w_{1}$, $\ldots$, $w_{n}$ are glued in left-to-right order ($w_{1}$ leftmost, etc.) to the bottom of $T_{2}$, and no other wires are attached to the bottom of $T_{2}$, and 
\item the top label of $T_{2}$ is the same as the bottom label of $T_{1}$.
\end{enumerate}
In this case, the result of removing the transistors $T_{1}$ and $T_{2}$ and the wires $w_{1}$, $\ldots$, $w_{n}$, and then attaching the top contacts of $T_{2}$ to the bottom contacts of $T_{1}$ (in left-to-right order-preserving fashion) is called \emph{removing a dipole}. The inverse 
operation is called \emph{inserting a dipole}.

A diagram that contains no dipoles is called \emph{reduced}. 

Two diagrams are \emph{equivalent modulo dipoles} if one diagram can be obtained from the other by repeatedly inserting or removing dipoles. The relation ``equivalent modulo dipoles" is indeed an equivalence relation. 
\end{definition}

\begin{remark} \label{remark:reduced}
A standard argument using Newman's Lemma \cite{Newman} shows that each equivalence class modulo dipoles contains a unique reduced diagram. If $\Delta$ is a diagram, then we let $r(\Delta)$ denote the unique reduced representative in its equivalence class.
\end{remark}

\begin{definition}
(Braided Diagram Groups) Let $\mathcal{P} = \langle \Sigma \mid \mathcal{R} \rangle$ be a semigroup presentation; let $w \in \Sigma^{+}$. The set of all equivalence classes of $(w,w)$-diagrams over $\mathcal{P}$ (under the dipole equivalence relation) is a group under concatenation. It is denoted $D_{b}(\mathcal{P}, w)$ and called the \emph{braided diagram group over $\mathcal{P}$ based at $w$}.

The set of all equivalence classes of planar $(w,w)$-diagrams over $\mathcal{P}$ is a group under concatenation, denoted $D(\mathcal{P}, w)$. This is the \emph{[planar] diagram group over $\mathcal{P}$ based at $w$}.

Similarly, the group of all equivalence classes of annular $(w,w)$-diagrams over $\mathcal{P}$ is denoted $D_{a}(\mathcal{P}, w)$ and called the \emph{annular diagram group over $\mathcal{P}$ based at $w$}.
\end{definition} 

\subsection{The diagram complex} \label{subsection:diagramcomplex}

In this section, we will describe a CAT($0$) cubical complex $\widetilde{K}_{b}(\mathcal{P}, w)$, the \emph{diagram complex}, on which $D_{b}(\mathcal{P}, w)$ acts properly and isometrically. The groups $D(\mathcal{P},w)$ and $D_{a}(\mathcal{P},w)$ admit actions on similar complexes, $\widetilde{K}(\mathcal{P},w)$ and $\widetilde{K}_{a}(\mathcal{P},w)$, respectively. The definitions of the latter complexes may be obtained from the definition of $\widetilde{K}_{b}(\mathcal{P},w)$ simply by replacing all mentions of braided diagrams with planar or annular diagrams, respectively. We will therefore concentrate on the case of $\widetilde{K}_{b}(\mathcal{P}, w)$.

We continue (in this subsection and the next) to let $\mathcal{P} = \langle \Sigma \mid \mathcal{R} \rangle$ be an arbitrary semigroup presentation such that $(u,u) \notin \mathcal{R}$ (for all $u \in \Sigma^{+}$), and let $w \in \Sigma^{+}$.

\begin{definition} \label{definition:verticesofK}
(Vertices of $\widetilde{K}_{b}(\mathcal{P}, w)$) A vertex of $\widetilde{K}_{b}(\mathcal{P}, w)$ is an equivalence class of reduced braided $(w,\ast)$-diagrams over $\mathcal{P}$; the equivalence relation in question is
\[ \Delta_{1} \sim \Delta_{2} \quad \Leftrightarrow \quad \Delta_{1} \circ \Pi = \Delta_{2}, \]
where $\Pi$ is a diagram with no transistors (i.e., a \emph{permutation diagram}).

We can think of a vertex as simply a braided $(w,\ast)$-diagram such that the bottommost wires do not attach to the bottom of the frame. We may thus sometimes refer to a reduced braided $(w,\ast)$-diagram $\Delta$ over $\mathcal{P}$ as a vertex itself, even though (as above) a vertex is technically an equivalence class of diagrams. We hope that this causes no confusion.
\end{definition}

\begin{definition}
(Cubes of $\widetilde{K}_{b}(\mathcal{P}, w)$) A marked cube in $\widetilde{K}_{b}(\mathcal{P}, w)$ is determined by a pair $(\Delta, \Psi)$, where $\Delta$ is a reduced braided $(w,v)$-diagram (for some $v \in \Sigma^{+}$), and $\Psi$ is a thin braided $(v,u)$-diagram (for some $u \in \Sigma^{+}$). Here, a \emph{thin diagram} $\Psi$ is such that no two transistors of $\Psi$ are comparable in the partial order $<$ on transistors (as defined in Definition \ref{definition:diagram}).

\end{definition}

\begin{definition}
(Realization of a marked cube) Given a marked cube $(\Delta, \Psi)$, we define its 
\emph{realization} $| \Delta, \Psi |$ as follows. If $\Psi$ contains $n$ transistors, then choose a numbering of the transistors $1, \ldots, n$. A corner of the cube $[0,1]^{n}$ is (obviously) identified with a binary string of length $n$. We specify a labelling of each corner $w = (a_{1}, \ldots, a_{n})$
of $[0,1]^{n}$ by a vertex of $\widetilde{K}_{b}(\mathcal{P}, w)$ as follows:
\begin{itemize}
\item if $a_{i} = 0$, then we remove the $i$th transistor of $\Psi$ by clipping the wires above it;
\item if $a_{i} = 1$, then we leave the $i$th transistor alone.
\end{itemize}
If we let $\Psi_{w}$ denote the result of performing the above operations, then
$r( \Delta \circ \Psi_{w})$ is the label of the corner $w$. 

This labelling of the vertices is uniquely determined by the pair $(\Delta, \Psi)$ and the numbering of the transistors.
\end{definition}

\begin{figure}[h]
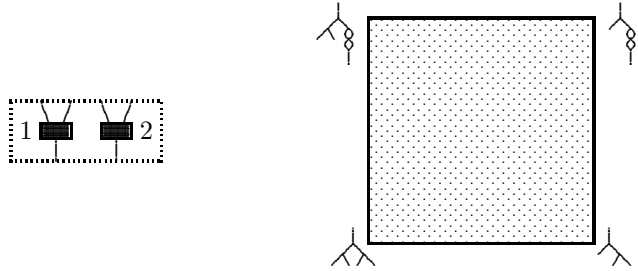
 
\vbox{\hspace{1cm}\beginpicture
\setcoordinatesystem units <.1 cm,.1cm> point at 0 0
\setplotarea x from -2 to 2, y from -2 to 2
\linethickness=.9pt

\setdots <2 pt >
\putrectangle corners at -30 4 and -50 -4

\setsolid
\setshadegrid span <.3pt>
\shaderectangleson
\putrectangle corners at -34 1 and -38 -1
\put {$1$} at -48 0
\put {$2$} at -32 0
\putrectangle corners at -42 1 and -46 -1

\plot -36 -1 -36 -4 /
\plot -44 -1 -44 -4 /

\setquadratic
\plot -34 4    -34.25 3.0625    -34.5 2.5 / 
\plot -34.5 2.5     -34.75 1.9375    -35 1 / 
\plot -38 4    -37.75 3.0625   -37.5 2.5 /
\plot -37.5 2.5     -37.25 1.9375  -37 1 / 
\plot -46 4    -45.75 3.0625    -45.5 2.5 /
\plot -45.5 2.5     -45.25 1.9375    -45 1 / 
\plot -42 4    -42.25 3.0625   -42.5 2.5 /
\plot -42.5 2.5     -42.75 1.9375  -43 1 / 

\setsolid
\setshadegrid span <2pt>
\putrectangle corners at 27.5 15 and -2.5 -15

\setlinear
\plot -4.5 -13  -4.5 -15 / 
\plot -4.5 -15   -7.3 -17.8 /
\plot -4.5 -15   -1.7 -17.8 /
\plot -5.05 -17.8   -5.9 -16.4 /
\plot -3.95 -17.8   -3.1 -16.4 / 

\plot 29.5 -13  29.5 -15 / 
\plot 29.5 -15   32.3 -17.8 /
\plot 29.5 -15  28.1 -16.4 /
\plot 30.05 -17.8   30.9 -16.4 / 

\plot -6.5 17  -6.5 15 / 
\plot -6.5 15   -9.3 12.2 /
\plot -6.5 15   -5.1 13.6 /  
\plot -5.1 10.8 -5.1 8.8 /  
\plot -7.05 12.2   -7.9 13.6 / 

\plot 31 17  31 15 / 
\plot 31 15   32.4 13.6 /
\plot 29.6 13.6    31 15 / 
\plot 32.4 10.8    32.4  8.8 / 

\setquadratic
\plot -5.1 13.6    -5.6 12.9    -5.1  12.2 / 
\plot -5.1 13.6    -4.6 12.9    -5.1 12.2 / 
\plot -5.1 12.2    -5.6 11.5    -5.1  10.8 / 
\plot -5.1 12.2    -4.6 11.5    -5.1 10.8 / 

\plot 32.4 13.6    31.9 12.9    32.4  12.2 / 
\plot 32.4 13.6    32.9 12.9    32.4 12.2 / 
\plot 32.4 12.2    31.9 11.5    32.4  10.8 / 
\plot 32.4 12.2    32.9 11.5    32.4 10.8 / 


\endpicture}\hfil
\caption{A thin braided $(x^{4}, x^{2})$-diagram  $\Psi$ over $\mathcal{P} = \langle x \mid x = x^{2} \rangle$ with a numbering of its transistors (left) and the associated marked square $(\Delta_{1}, \Psi)$ (right). Note that $\Delta_{1}$ is as in Figure \ref{figure:1}.}
\label{figure:2}
\end{figure}

\begin{example}
Consider the pair $(\Delta_{1}, \Psi)$, where $\Delta_{1}$ is as in Figure \ref{figure:1} and $\Psi$ appears in the left side of Figure \ref{figure:2}. Both are diagrams over $\mathcal{P} = \langle x \mid x = x^{2} \rangle$. The associated marked cube appears in the right side of Figure \ref{figure:2}. Note that we have used a modified  notation to label the corners of the square. In particular, we have shrunk each transistor to a single vertex; this is harmless, since the defining data of the braided diagrams can still be read from the picture. 

The frames have been omitted entirely. By Definition \ref{definition:verticesofK}, the left-right ordering of the wires attached to the bottom of each frame is irrelevant, while the left-right ordering of the wires at the top of each frame is obvious, so the omission of the frames is also harmless.

We note finally that the figure-eights occurring in the top left and top right corners depict pairs of transistors with crossing wires in between (compare to the right-hand picture in Figure \ref{figure:1}). In particular, the middle ``vertices" in these figure-eights are wire crossings, not transistors. 
\end{example}

\begin{definition} (The diagram complex $\widetilde{K}_{b}(\mathcal{P},w)$)
Pick a realization $| \Delta, \Psi |$ for each cube $(\Delta, \Psi)$ as above. We glue two such realizations of cubes together along faces whose vertices have the same labelling. 
\end{definition}

\begin{proposition} \label{proposition:orbit}
The result of the above gluing is a CAT(0) cubical complex \cite{Far(th), Far(pg)}, and thus contractible. There is a natural group action of $D_{b}(\mathcal{P}, w)$ on 
$\widetilde{K}_{b}(\mathcal{P}, w)$ by the rule
\[ \widetilde{\Delta} \cdot ( \Delta, \Psi ) = ( r( \widetilde{\Delta} \circ \Delta), \Psi ). \]
[Here $r(\Delta)$ is as defined in Remark \ref{remark:reduced}.] This action is proper and by isometries. It is not necessarily cocompact. In fact, if $\Delta_{1}$ and $\Delta_{2}$ are vertices, then $\Delta_{1}$ and $\Delta_{2}$ lie in the same orbit if and only if the bottom labels of the diagrams $\Delta_{1}$ and $\Delta_{2}$ are permutations of each other. \qed
\end{proposition}

The proof that the complexes in question are CAT($0$) uses Gromov's famous link condition. We refer the reader to \cite{BH} for one standard account.

\subsection{The link of a vertex in the diagram complex}

\begin{definition} \label{definition:disjointapp}(Disjoint applications of relations)
Let $\Delta_{1}$ and $\Delta_{2}$ be thin braided $(w, \ast)$-diagrams over $\mathcal{P}$. 
Let $w_{1}, \ldots, w_{n}$ be the wires of $\Delta_{1}$ that meet the top of the frame, listed in the left-to-right order in which they meet the top of the frame. Let $\hat{w}_{1}, \ldots, \hat{w}_{n}$ be the wires of $\Delta_{2}$ that meet the top of its frame, ordered similarly. (Note that the number of such wires is the same in both cases, since both numbers are equal to the length of $w$.)

Let
\[ S_{1} = \{ j \in \{ 1, \ldots, n \} \mid \text{the bottom of } w_{j} \text{ is attached to a transistor of }
\Delta_{1} \} \]
and 
\[ S_{2} = \{ j \in \{ 1, \ldots, n \} \mid \text{the bottom of } \hat{w}_{j} \text{ is attached to a transistor of }
\Delta_{2} \}. \]
We say that $\Delta_{1}$ and $\Delta_{2}$ represent \emph{disjoint applications of relations to $w$}, or simply that $\Delta_{1}$ and $\Delta_{2}$ are \emph{disjoint}, if $S_{1} \cap S_{2} = \emptyset$.
\end{definition}

\begin{proposition} \cite{Far(th), Far(pg)} \label{proposition:linkdescription}(Description of the link)
Let $v$ be a vertex of $\widetilde{K}_{b}(\mathcal{P}, w)$; let $\Delta$ represent $v$. Assume $\Delta$ is a braided $(w, \hat{w})$-diagram. Define an abstract simplicial complex $L(v)$ as follows: The vertices are the braided $(\hat{w}, \ast)$-diagrams over $\mathcal{P}$ that contain exactly one transistor. A finite collection of such vertices spans a simplex if and only if the vertices are pairwise disjoint (in the sense of Definition \ref{definition:disjointapp}). 

The link of $v$ in $\widetilde{K}_{b}(\mathcal{P}, w)$ is isomorphic to $L(v)$. \qed
\end{proposition}

\subsection{The links associated to $\mathbf{\mathcal{P} = \langle x, a \mid 
xax = x \rangle}$}
The links in the cubical complex $\widetilde{K}_{b}(\mathcal{P}, w)$, where $\mathcal{P} = \langle x, a \mid xax = x \rangle$ and $w \in \{ x, a \}^{+}$, will be especially important in our main argument. We offer a direct description here. 

\begin{definition} \label{definition:links} (Abstract links in $\widetilde{K}_{b}(\mathcal{P},w)$)
Let $v = x^{k}a^{\ell} \in \{ x, a \}^{+}$ ($k, \ell \geq 0$). We let 
\[ D_{v} = \{ (m,n,p) \mid m,n \in \{1, \ldots, k \}; p \in \{1, \ldots, \ell \}, m \neq n \}, \]
and
\[ A_{v} = \{ (m) \mid m \in \{1, \ldots, k \} \}. \]
We define an abstract simplicial complex $\mathrm{lk}(v)$ as follows. The vertex set of $\mathrm{lk}(v)$ is $D_{v} \cup A_{v}$. A non-empty
collection 
\[ S = \{ (m_{1}, n_{1}, p_{1}), \ldots, (m_{\alpha}, n_{\alpha}, p_{\alpha}) \} \cup \{ (m'_{1}), \ldots, (m'_{\beta}) \}
\]
is a simplex in $\mathrm{lk}(v)$ if and only if there are no repetitions
in the lists
\[ m_{1}, n_{1}, \ldots, m_{\alpha}, n_{\alpha}, m'_{1}, \ldots,
m'_{\beta} \]
and
\[ p_{1}, \ldots, p_{\alpha}. \]
(We note that one of $\alpha$ or $\beta$ may be $0$, but not both.)

We say that $\mathrm{lk}(v)$ is the \emph{(abstract) link of the word $v$} in $\widetilde{K}_{b}(\mathcal{P}, w)$.
\end{definition}

\begin{definition} \label{definition:descendinglinks} (Abstract descending links in $\widetilde{K}_{b}(\mathcal{P},w)$) 
Let $v$ and $D_{v}$ be as above. The vertex set
of $\mathrm{lk}_{\downarrow}(v)$ is $D_{v}$. A collection of vertices $S$ determines a simplex exactly under the conditions specified in Definition \ref{definition:links}.

We say that $\mathrm{lk}_{\downarrow}(v)$ is the \emph{(abstract) descending link of the word $v$} in $\widetilde{K}_{b}(\mathcal{P}, w)$.
\end{definition}

\begin{proposition} \label{proposition:descriptionoflink}
Let $v = x^{k}a^{\ell}$. The simplicial complex $\mathrm{lk}(v)$ is isomorphic to the link of $v$ in the complex 
$D_{b}(\mathcal{P}, w)$, where $\mathcal{P} = \langle x, a \mid xax = x \rangle$ and $w = x^{k+m}a^{\ell+m}$, for any $m \in \mathbb{Z}$ such that $k+m, \ell + m \geq 0$. 

The simplicial complex $\mathrm{lk}_{\downarrow}(v)$ is the full subcomplex of $\mathrm{lk}(v)$ determined by the $(xax,x)$-transistors.
\end{proposition}

\begin{proof}
The first statement is a special case of Proposition \ref{proposition:linkdescription}. The second statement follows easily from the identification of the link of $v$ with 
$L(v)$ from Proposition \ref{proposition:linkdescription}.
See also Remark \ref{remark:link} below for more details.
\end{proof}

\begin{remark} \label{remark:link}
The complexes $\mathrm{lk}(v)$ and $\mathrm{lk}_{\downarrow}(v)$, as defined in Definitions \ref{definition:descendinglinks} and \ref{definition:links}, are both flag complexes, as may be easily checked. The latter fact is used in the proof that the larger cubical complexes are CAT($0$) (see Proposition \ref{proposition:orbit}).

The members of $D_{v}$ represent ``descending" applications of relations (i.e., descending attachments of transistors, in which three contacts are on top of the transistor). The numbers $m$ and $n$ in the triple $(m,n,p)$ describe how the left and right top ``$x$"-contacts  of the transistor are connected to the top of the frame, where the number $m$ (for instance) indicates that the top left ``$x$"-contact of the transistor is connected by a wire to the $m$th ``$x$''-contact from the left at the top of the frame. The third coordinate, $p$, describes how the top ``$a$''-contact of the transistor is connected to the top of the frame (namely, at the $p$th ``$a$''-contact from the left). 

In a similar way, the members of $A_{v}$ represent ``ascending" applications of relations, in which a single ``$x$"-contact appears at the top of a transistor. 
\end{remark}

\section{The groups $QF$, $QT$, and $QV$} \label{section:QFQTQV}

In this section, we describe the groups $QF$, $QT$, and $QV$, first as groups of quasi-automorphisms, and second as braided diagram groups over semigroup presentations (or as subgroups of such groups). 

\subsection{The groups $QF$, $QT$, and $QV$ as quasi-automorphism groups}

The results in this subsection are taken from \cite{NJG} and included for the reader's convenience. Most of the definitions are from Section 2 of that source.

\begin{definition} \label{definition:QV}
($QV$) We let $\mathcal{T}$ denote the infinite rooted binary tree. The vertices of $\mathcal{T}$ may be identified with the members of the monoid $\{ 0, 1 \}^{\ast}$, which consists of all finite binary sequences, including the empty sequence, which we denote $\epsilon$.

For a given $v \in \mathcal{T}^{0}$, we regard $v0$ as the left child of $v$ and $v1$ as the right child.

A bijection $h: \mathcal{T}^{0} \rightarrow \mathcal{T}^{0}$ is a member of $QV$ if, for almost all $v \in \mathcal{T}^{0}$, $h$ sends the left (respectively, right) child of $v$ to the left (respectively, right) child of $h(v)$. (Here, ``almost all" means ``with at most finitely many exceptions".) Note, in particular, that $h$ need not preserve adjacency.

The set $QV$ is a group under composition of functions.
\end{definition}

\begin{proposition}
(Tree pair representatives)
Any $h \in QV$ can be specified by a pair $( (T_{1}, \sigma, T_{2}), f)$ where:
\begin{enumerate}
\item $(T_{1}, \sigma, T_{2})$ is the usual tree pair representative of an element of $V$. Thus, $T_{1}$ and $T_{2}$ are finite rooted ordered binary trees with the same number of leaves, and $\sigma$ is a bijection of the leaves, and 
\item $f$ is a bijection between the set of interior nodes of $T_{1}$ and the set of interior nodes of $T_{2}$. (Here, an \emph{interior node} of $T_{i}$ is a vertex of $T_{i}$ that is not a leaf.) \qed
\end{enumerate}
\end{proposition}

\begin{figure}[h]
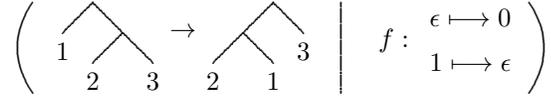
 
\vbox{ \hspace{1cm}\beginpicture
\setcoordinatesystem units <.2cm, .2cm> point at 0 0
\setplotarea x from -2 to 2, y from -2 to 2
\linethickness=.9pt

\plot -10 2    -12 0 /
\plot -8 0     -10 -2 /
\plot -10 2     -6 -2 /

\put {$1$} at -12 -1.25
\put {$2$} at -10 -3.25
\put {$3$} at -6 -3.25
\put {$\rightarrow$} at -4 0 
\put {$1$} at 2 -3.25
\put {$2$} at -2 -3.25
\put {$3$} at 4 -1.25
\plot 6.5 2   6.5 -4 /

\plot -2 -2  2 2  /
\plot 0 0    2 -2 /
\plot 2 2    4 0 /

\put {$\epsilon \longmapsto 0$} at 15.1 1
\put {$1 \longmapsto \epsilon$} at 15.1 -2
\put {$f:$} at 10 -.5

\setquadratic
\plot -14 -4   -15 -1    -14 2 / 
\plot 19 -4   20 -1  19 2 / 

\endpicture}
\caption{An element $h$ of $QV$, described by a tree pair (left) and a bijection $f$ of interior nodes (right).}
\label{figure:3}
\end{figure}


\begin{example}
For instance, consider the pair in Figure \ref{figure:3}. Here we have indicated the bijection $\sigma$ of leaves by direct enumeration. Thus, the leaf $0$ goes to the leaf $01$, $10$ to $00$, and $11$ to $1$. The resulting assignment determines a bijection
$h$ between the vertices of $\mathcal{T}$ that are not interior nodes of $T_{1}$ and the vertices of $\mathcal{T}$ that are not interior nodes of $T_{2}$ once we specify that $h$ takes the left and right children of any non-interior vertex of $T_{1}$ to the left and right children (respectively) of its image. Thus, 
\[ h(010) = \sigma(0)10 = 0110. \]
The bijection $f$ specifies the values of $h$ on the remaining vertices.
\end{example}

\begin{proposition} \label{proposition:QVpi}
The function $\pi: QV \rightarrow V$ determined by projection onto the first coordinate is a surjective homomorphism; the kernel is isomorphic to the group $S_{\infty}$ of self-bijections of $\mathbb{N}$ having finite support.
 \qed
\end{proposition}

We easily define $QF$ and $QT$ in terms of $\pi$:

\begin{definition} \label{definition:QFQT}
(Definitions of $QF$ and $QT$)
\begin{enumerate}
\item $QF = \pi^{-1}(F)$;
\item $QT = \pi^{-1}(T)$.
\end{enumerate}
\end{definition}

\subsection{The groups $QF$, $QT$, and $QV$ as diagram groups} \label{subsection:QVdiagramgroup}

The observation that $QV$ is a braided diagram group first appeared in 
\cite{FH2}; see Example 4.4 from that source. Here we will develop the idea in somewhat more depth.

\begin{proposition}
The group $QV$ is isomorphic to $D_{b}(\mathcal{P}, x)$, where
$\mathcal{P} = \langle x, a \mid x = xax \rangle$. 
\end{proposition}

\begin{proof}
We sketch the isomorphism. Let $\Delta$ be a braided $(x,x)$-diagram over 
$\mathcal{P}$. We can classify each $T \in \Delta$ as either positive or negative, as follows:
a transistor is \emph{positive} if the top label is $x$ and the bottom label is $xax$; it is \emph{negative} in the opposite case. If a positive transistor of $\Delta$ is less than a negative transistor in the partial order on transistors, then $\Delta$ is necessarily not reduced. We can then remove dipoles until no positive transistor is less than a negative transistor. The diagram $\Delta$ can now be expressed in the form $\Delta = \Delta_{1} \circ \Delta_{2}^{-1}$, where all transistors in $\Delta_{1}$ and $\Delta_{2}$ are positive. We call such braided $(x, \ast)$-diagrams $\Delta_{1}$ and $\Delta_{2}$ \emph{positive}.

A given positive braided $(x, x^{n+1}a^{n})$-diagram $(n \geq 0)$, read from bottom to top, can be interpreted as a set of instructions for assembling an infinite binary tree (represented by ``$x$") out of $n+1$ infinite binary trees and $n$ vertices (represented by ``$a$"). Specifically, a positive transistor selects two infinite binary trees (two ``$x$"s) and a single vertex (an ``$a$") and combines them into a single binary tree (an ``$x$"). (Or: the transistor encodes the action of assembling these two trees and a vertex into a single binary tree.) The bottom left ``$x$" contact becomes the rooted subtree with root $0$ in the new tree; the bottom right ``$x$" contact becomes the rooted subtree with root $1$, and the vertex represented by ``$a$" becomes the root. Conversely, a negative transistor (read from bottom to top) represents a  dissection of a binary tree into three pieces.

Assume, without loss of generality, that $\Delta_{1}$ and $\Delta_{2}$ are braided $(x, x^{n+1}a^{n})$-diagrams. It follows from the above discussion that $\Delta_{2}^{-1}$ represents a dissection of the standard infinite binary tree into $n+1$ binary trees and $n$ vertices; $\Delta_{1}$ represents the subsequent reassembly of these pieces into a single binary tree. (Thus, $\Delta_{2}^{-1}$ is a dissection of the domain tree, and $\Delta_{1}$ shows how to reassemble the pieces into the range tree.) We let $f_{\Delta} \in QV$ denote the function so determined by the diagram $\Delta$.

The proof is completed by noting that
\begin{enumerate}
\item the function $f_{\Delta}$ indicated above is always in $QV$, and any element of $QV$ is
$f_{\Delta}$ for appropriate $\Delta$;
\item if $\Delta'$ and $\Delta''$ are braided $(x,x)$-diagrams over $\mathcal{P}$, and $\Delta'$ and $\Delta''$ are equivalent modulo dipoles, then $f_{\Delta'} = f_{\Delta''}$, and
\item the indicated correspondence is a homomorphism; i.e., $f_{\Delta' \circ \Delta''} = 
f_{\Delta'} \circ f_{\Delta''}$, for all braided $(x,x)$-diagrams $\Delta'$, $\Delta''$ over $\mathcal{P}$.
\end{enumerate}
\end{proof}

\begin{figure}[h]
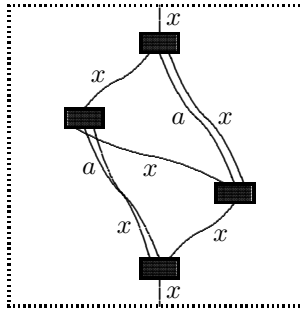
 
\vbox{\hspace{1cm}\beginpicture
\setcoordinatesystem units <.125 cm,.125cm> point at 0 0
\setplotarea x from -2 to 2, y from -2 to 2
\linethickness=.9pt


\setdots <2 pt >
\putrectangle corners at 16 16 and -16 -16

\setsolid
\setshadegrid span <.3pt>
\shaderectangleson

\putrectangle corners at 2 13 and -2 11 
\putrectangle corners at -10 5  and -6 3  
\putrectangle corners at 6 -3 and 10 -5 
\putrectangle corners at -2 -11 and 2 -13 


\plot 0 16    0 13 /
\plot 0 -16  0  -13 / 

\put {$x$} at 1.5 14.5 
\put {$x$} at 1.5 -14.5 
\put {$x$} at 7 4 
\put {$a$} at 2 4 
\put {$x$} at -6.5 8.5 
\put {$x$} at -1 -1.5 
\put {$x$} at 6.5 -8.5 
\put {$x$} at -3.75 -7.625 
\put {$a$} at -7.5 -1.375

\setquadratic
\plot 1 11    3 6.625     5 4 /
\plot 5 4     7 1.375     9 -3 / 
\plot 0 11    2 6.625     4 4 /
\plot 4 4     6 1.375     8 -3 / 
\plot 1 -11     2.75 -9.125     4.5 -8 /
\plot 4.5 -8     6.25 -6.875    8 -5 / 
\plot 7 -3    3 -1.125    -1 0 /
\plot -1 0    -5 1.125    -9 3 / 
\plot -1 11     -2.75 9.125    -4.5 8 /
\plot -4.5 8     -6.25 6.875    -8 5 / 
\plot -7 3    -5.5 -1.375     -4 -4 /
\plot -4 -4     -2.5 -6.625    -1 -11 / 
\plot 0 -11     -2 -6.625    -4 -4 /
\plot -4 -4    -6 -1.375   -8 3 / 

\endpicture}\hfil
\caption{The element $h \in QV$ from Figure \ref{figure:3}, described in diagram form.}
\label{figure:4}
\end{figure}


\begin{example} Figure \ref{figure:4} shows how the element $h$ of $QV$ from Figure \ref{figure:3} may be specified by a braided $(x,x)$-diagram over the presentation $\mathcal{P} = \langle a, x \mid xax = x \rangle$. Here, the negative transistors represent a dissection of the domain tree and the positive transistors represent a dissection of the range tree.

We read the diagram from the bottom up. The bottommost transistor represents the dissection of the basic binary tree $\mathcal{T}$ into three pieces: the binary tree rooted at ``$0$"
(represented by the leftmost wire at the top of the transistor; hereafter, we simply call this the $0$-tree), the root (represent by the ``$a$" wire), and $1$-tree (represented by the rightmost wire). We note that the rightmost wire leads to another negative transistor; this transistor represents the dissection of the subtree rooted at $1$ into three more pieces: the $10$-tree (represented by the left wire), the root $1$ (represented by the ``$a$" wire), and the $11$-tree (represented by the right wire). There are no more negative transistors, so the domain tree is dissected into five pieces (in left-to-right order): the $0$-tree, the root $\epsilon$, the $10$-tree, the vertex $1$, and the $11$-tree. 

In a similar way, the dissection of the range tree is encoded by the positive transistors, which are read from the top down. Thus,
the range tree is dissected into (respectively): the $00$-tree, the vertex $0$, the $01$-tree, the root $\epsilon$, and the $1$-tree.

The mapping between these pieces is determined by wire connections; in particular, we have that $h$ sends the $0$-tree to the $01$-tree; the root $\epsilon$ to $0$, the $10$-tree to the $00$-tree, the vertex $1$ to $\epsilon$, and the $11$-tree to the $1$-tree.
\end{example}

\subsection{Actions of $QF$, $QT$, and $QV$ on CAT($0$) cubical complexes} \label{subsection:CAT0cubical}

In this subsection, we will describe actions of $QF$, $QT$, and $QV$ on CAT($0$) cubical complexes. All of the actions under consideration will be proper and by isometries. We have already seen an action of $QV$ with the desired properties: indeed $QV \cong D_{b}(\mathcal{P},x)$ (where $\mathcal{P} = \langle a, x \mid x = xax \rangle$), so $QV$ acts properly and isometrically on the CAT($0$) cubical complex
$\widetilde{K}_{b}(\mathcal{P}, x)$. As subgroups of $QV$, $QF$ and $QT$ also act properly and isometrically on $\widetilde{K}_{b}(\mathcal{P}, x)$. We will, however, want $QF$ and $QT$ to act on more economical complexes when we attempt to establish the $F_{\infty}$ property for these groups. Here we will find such complexes as convex subcomplexes of $\widetilde{K}_{b}(\mathcal{P}, x)$.

\begin{proposition} (Diagram description of elements in $QF$ and $QT$)
Let $\Delta$ be a braided $(x,x)$-diagram over $\mathcal{P} = \langle a, x \mid x = xax \rangle$.
\begin{enumerate}
\item The diagram $\Delta$ represents an element of $QF$ if and only if the result of deleting each edge labelled ``$a$" results in a planar $(x,x)$-diagram over the presentation $\mathcal{P}' 
= \langle x \mid x = x^{2} \rangle$.
\item The diagram $\Delta$ represents an element of $QT$ if and only if the result of deleting each edge labelled ``$a$" results in an annular $(x,x)$-diagram over the presentation $\mathcal{P}'$. \qed
\end{enumerate}
\end{proposition}

The above descriptions of $QF$ and $QT$ suggest the following definition:

\begin{definition}
(The complexes $K_{QF}$ and $K_{QT}$) Let $\mathcal{P} = \langle a, x \mid x = xax \rangle$. We denote the complex
$\widetilde{K}_{b}(\mathcal{P}, x)$ by $K_{QV}$.

For a braided $(x,\ast)$-diagram $\Delta$ over $\mathcal{P}$, let $\pi(\Delta)$ denote the diagram over
$\mathcal{P}' = \langle x \mid x = x^{2} \rangle$ obtained by deleting each wire labelled by ``$a$''. [We note that the map 
$\pi$ determines a homomorphism from $D_{b}(\mathcal{P}, x)$
to $D_{b}(\mathcal{P}', x) \cong V$, where $\mathcal{P}' = \langle x \mid x = x^{2} \rangle$. Indeed, this is the same homomorphism as the one in Proposition \ref{proposition:QVpi}.] 

We define subcomplexes $K_{QF}$, $K_{QT}$ of $K_{QV}$ to be the full subcomplexes of $K_{QV}$ determined by vertex 
sets $K_{QF}^{0}$ and $K_{QT}^{0}$, which are defined as follows: a vertex $v$ of $K_{QV}$ is in $K_{QF}^{0}$ (respectively, in $K_{QT}^{0}$) if and only if it has a representative $\Delta$ such that $\pi(\Delta)$ is planar (respectively, annular).
(Recall that a vertex is an equivalence class of braided diagrams (see Definition \ref{definition:verticesofK}), so a representative of $v$ is a choice of diagram from the equivalence class.)
\end{definition}

\begin{proposition} \label{proposition:CAT0cubical}
The complexes $K_{QF}$ and $K_{QT}$ are path-connected subcomplexes of $K_{QV}$ such that, for each 
$v \in K_{QF}^{0}$ (respectively, $K_{QT}^{0}$) the link of $v$ in $K_{QF}$ (respectively, $K_{QT}$) embeds in the link of $v$ in $K_{QV}$ as a full subcomplex. 

In particular, $K_{QF}$ and $K_{QT}$ are CAT($0$) cubical complexes. Moreover, $QF$ and $QT$ act properly and isometrically on $K_{QF}$ and $K_{QT}$. 
\end{proposition}

\begin{proof}
We first show that $K_{QF}$ and $K_{QT}$ are path-connected subcomplexes of $K_{QV}$. Since the proofs in both cases are similar, it will be sufficient to consider $K_{QF}$.

 Let $v$ be a vertex in $QF$. We will show that $v$ can be connected to the natural basepoint $v^{\ast}$ of $K_{QF}$, which is the equivalence class of the permutation $(x,x)$-diagram $\Delta'$. (Recall that a permutation diagram has no transistors (see Definition \ref{definition:verticesofK}). Thus, $\Delta'$ consists simply of a frame and a single wire, which runs from the top of the frame to the bottom.) We prove that $v$ can be connected to $v^{\ast}$ by an edge-path using induction on the number $n$ of transistors in a diagram representative $\Delta$ for $v$. This is trivial if $n=0$. If $n \geq 1$, then we pick a transistor $T$ of $\Delta$ that is minimal in the partial order on transistors. We let $\Delta_{1}$ denote the diagram that is obtained from deleting $T$ (and all depending wires) from $\Delta$, and we let $v_{1}$ denote the vertex represented by $\Delta_{1}$. Clearly,
$v_{1}$ and $v$ are adjacent in $K_{QF}$, and, by induction, $v_{1}$ can be connected to the basepoint $v^{\ast}$ by a path. It follows that $v$ can also be so connected to $v^{\ast}$, and it then follows that $K_{QF}$ is path-connected.

To show that $K_{QF}$ embeds as convex subspace of $K_{QV}$, it now suffices to prove that the link of an arbitrary vertex $v \in K_{QF}$ includes into the link of $v \in K_{QV}$
as a full subcomplex. (We are appealing to Theorem 1(2) from \cite{CW}.) Assume that $\Delta$ represents $v$, and the bottom label of $\Delta$ is
$x^{k}a^{\ell}$, for some integers $k$ and $\ell$, with $k > 0$. Let $\Psi_{1}$, $\Psi_{2}$, $\ldots$, $\Psi_{m}$ be $(x^{k}, a^{\ell}, \ast)$ diagrams over the semigroup presentation $\mathcal{P}$ be such that: i) each $\Psi_{i}$ has exactly one transistor; ii) the product $\Delta \circ \Psi_{i}$ determines a vertex of
$D(\mathcal{P}, x)$, possibly after reducing a dipole; iii) any two $\Psi_{i}$, $\Psi_{j}$ represent disjoint applications of relations to $x^{k}a^{\ell}$ (in the sense of Definition \ref{definition:disjointapp}). We note that i)-iii) are equivalent to the condition that the $\Psi_{i}$ are pairwise joined by edges in the link $\mathrm{lk}^{QF}(v)$. We are required to show that the collection of all of the $\Psi_{i}$ spans a single $(m-1)$-simplex in $\mathrm{lk}^{QF}(v)$. But the condition that the $\Psi_{i}$ are pairwise disjoint shows that
\[  \{ \Psi_{1}, \ldots, \Psi_{m} \} \]
is a simplex in $\mathrm{lk}^{QV}(v)$, and it is easy to see that the ``union" of the $\Psi_{i}$ (see Remark \ref{remark:link})
is planar after the application of $\pi$. This means that the link $\mathrm{lk}^{QF}(v)$ is a full subcomplex of $\mathrm{lk}^{QV}(v)$. It now follows that $K_{QF}$ is a convex subset of $K_{QV}$, and therefore CAT($0$).

Since the action of $QF$ on $K_{QF}$ is determined by stacking diagrams, and such stacking preserves planarity under the projection $\pi$, $QF$ acts on $K_{QF}$ by cell-permuting automorphisms, and thus by isometries. The properness of the action follows from the properness of the action of $QF$ on $K_{QV}$. 
\end{proof} 

\section{The $F_{\infty}$ Property} \label{section:Finfinity}

In this section, we will prove that each of the groups $QF$, $QT$, and $QV$ has type $F_{\infty}$. Our proof uses Ken Brown's well-known finiteness criterion , which we recall below in Subsection \ref{subsection:overview} (Theorem \ref{theorem:Brown}). In the remaining subsections we establish the various hypotheses of Theorem \ref{theorem:Brown}. The section ends in Subsection \ref{subsection:conclusion} with a proof that each of the groups $QF$, $QT$, and $QV$ has the $F_{\infty}$ property.

\subsection{Brown's Finiteness Criterion} \label{subsection:overview}

Our proof that $QF$, $QT$, and $QV$ have type $F_{\infty}$ uses a well-known theorem due to Ken Brown.

\begin{theorem}  \label{theorem:Brown} \cite{Brown}
(Brown's Finiteness Criterion.) Let $X$ be a CW-complex. Let $G$ be a group acting on $X$. If 
\begin{enumerate}
\item $X$ is contractible;
\item $G$ acts cellularly on $X$, and
\item there is a sequence of subcomplexes $X_1 \subseteq X_2 \subseteq ... \subseteq X_n \subseteq \ldots \subseteq X$ such that
\begin{enumerate}
\item $X = \bigcup_{n=1}^{\infty} X_{n}$;
\item $G$ acts cocompactly on $X_i$ and leaves each $X_{i}$ invariant;
\item $G$ acts with finite cell stabilizers, and
\item for every $k \geq 0$, there exists an $N$ such that $X_n$ is $k$-connected for every $n \geq N$
\end{enumerate}
\end{enumerate}
then $G$ is of type $F_\infty$. \qed
\end{theorem}

We note that $QF$, $QT$, and $QV$ each act cellularly on contractible CW-complexes by the results of Subsection \ref{subsection:CAT0cubical}. It therefore remains to introduce suitable filtrations $\{ X_{i} \mid i \in \mathbb{N} \}$ that satisfy (3).  

\subsection{Filtrations by subcomplexes} \label{subsection:filtration}

\begin{definition} If $\Sigma$ is an alphabet, $p \in \Sigma$, and $w \in \Sigma^{+}$, then we let $|w|_{p}$ denote the number of occurrences of the letter $p$ in $w$. 
\end{definition}

\begin{definition} \label{definition:filtration}
Let $X = K_{QF}$, $K_{QT}$, or $K_{QV}$. For $n \geq 1$, let $L_{n}$ denote the collection of words $w$ in the alphabet $\{ a, x \}$ such that
$|w|_{a} = j-1$ and $|w|_{x} = j$, for some 
$j \in  \{ 1, \ldots, n \}$.

 We let $X_{n}$ denote the subcomplex generated by the collection of all vertices of $X$ having labels in $L_{n}$.
Thus, $X_{n}^{0}$ consists of all vertices having labels in $L_{n}$, and a higher-dimensional cube $C$ of $X$ lies in $X_{n}$ if and only if the label of each corner of $C$ lies in $L_{n}$.
\end{definition}  

\begin{proposition} \label{proposition:filtration}
If $X = K_{QF}$, $K_{QT}$, or $K_{QV}$, then $\{ X_{n} \mid n \in \mathbb{N} \}$ is a filtration of $X$ by subcomplexes that satisfies 3a), b), and c) from Theorem \ref{theorem:Brown}, where $G$ is $QF$, $QT$, or $QV$ (respectively).
\end{proposition}

\begin{proof}
We prove the proposition for $X = K_{QV}$. The proofs in the other two cases are similar.

We first prove that property 3a) holds. Let $C$ be any cube in $K_{QV}$. Each corner $v$ of the cube may be represented by a braided $(x, \ast)$-diagram $\Delta$. It is easy to see (by induction on the number of transistors in $\Delta$) that the bottom label of
$\Delta$ is $x^{j}a^{j-1}$, up to permutation of the letters, for some $j \geq 1$. Since $C$ has a finite number of corners, there is some particular $k \geq 1$ such that the bottom label of each corner lies in $L_{k}$. It follows that all corners of $C$ lie in $X_{k}$,
so that $C \subseteq X_{k}$. Thus, any cube of $X$ is contained in some 
$X_{k}$, proving 3a).

Next we establish property 3b). Recall that two vertices of $K_{QV}$ are in the same orbit if and only if they may be represented by diagrams with the same bottom labels (Proposition \ref{proposition:orbit}). It follows that each $X_{k}$ is invariant under the action of $QV$, and that there are only finitely many vertices of $X_{k}$ modulo the action of $QV$. Since the link of each vertex in $QV$ is finite (Proposition \ref{proposition:linkdescription} and 
Definition \ref{definition:links}), the action of $QV$ on each $X_{k}$ is cocompact. This proves 3b). 

The group $QV$ acts on each $X_{k}$ with finite cell stabilizers since $QV$ acts on $K_{QV}$ itself with finite cell stabilizers. This proves 3c).
\end{proof}

\subsection{Analysis of descending links} \label{subsection:analysis}

In this subsection, we consider the descending links in the complexes $K_{QF}$, $K_{QT}$, and $K_{QV}$. We first demonstrate (in \ref{subsubsection:connectivity}) that the connectivity of the complexes $X_{n}$ is determined by the connectivity of the descending links $\mathrm{lk}_{\downarrow}(w)$, where $w \in L_{n}$. We then compute the connectivity of $\mathrm{lk}_{\downarrow}(w)$ as a function of $n$ in \ref{subsubsection:links}.

\subsubsection{The descending link and connectivity of the complexes $X_{n}$} 
\label{subsubsection:connectivity}
Fix $n \geq 2$. Let $X$ denote any of the complexes $K_{QF}$, $K_{QT}$, or $K_{QV}$. (The current discussion applies equally to all three complexes, with inessential differences, so we will assume that $X = K_{QV}$ for simplicity.) For a small $\epsilon > 0$, consider the union of all $\epsilon$-neighborhoods around the vertices $v \in X_{n}$ having bottom labels $x^{n}a^{n-1}$. Let $A$ denote the complement of this union, and let $B$ denote its closure. Clearly, $X_{n} = A \cup B$.

Since a vertex $v \in X_{n}$ with bottom label $x^{n}a^{n-1}$ is adjacent in $X_{n}$ only to vertices with the bottom label $x^{n-1}a^{n-2}$, the link of $v$ in $X_{n}$ is isomorphic to the descending link
$\mathrm{lk}_{\downarrow}(v)$, as described in Definition \ref{definition:descendinglinks}. It follows that $A \cap B$ may be identified with a countable disjoint union of the links $\mathrm{lk}_{\downarrow}(v)$, as $v$ runs over all vertices in $X_{n}$ with bottom label $x^{n}a^{n-1}$. Moreover, the subspace $A$ strong deformation retracts onto the complex $X_{n-1}$.  It follows that, up to homotopy, we can describe $X_{n}$ as follows:
\[ X_{n} = X_{n-1} \bigcup_{\coprod_{v} \mathrm{lk}_{\downarrow}(v)} \left(\coprod_{v} C_{v} \right), \]
where the disjoint unions are over all vertices $v$ having bottom label $x^{n}a^{n-1}$, and the spaces $C_{v}$ are cones on the descending links, and therefore contractible.

Standard arguments using van Kampen's theorem, the Mayer-Vietoris sequence, and the Hurewicz theorem yield the following result.

\begin{proposition} \label{proposition:standard}
If the descending link $\mathrm{lk}_{\downarrow}(x^{k}a^{k-1})$ is $n$-connected ($n \geq 0$), then the inclusion map
$X_{k-1} \hookrightarrow X_{k}$ induces isomorphisms
$\pi_{j}(X_{k-1}) \rightarrow \pi_{j}(X_{k})$, for $0 \leq j \leq n$. 

In particular, if $\mathrm{lk}_{\downarrow}(x^{k}a^{k-1})$ is $n$-connected for all $k \geq N$ (for some $N \in \mathbb{N}$), then $X_{m}$ is $n$-connected, for all $m \geq N$.
\qed 
\end{proposition}

For future reference, we now give a combinatorial description of the descending links.

\begin{definition} \label{definition:QFQTlinks}
Let $v = x^{k}a^{\ell}$. We let $\mathrm{lk}_{\downarrow}^{QV}(v) = 
\mathrm{lk}_{\downarrow}(v)$ (as defined in Definition \ref{definition:descendinglinks}). 
We also let $\mathrm{lk}_{\downarrow}^{QF}(v)$ and $\mathrm{lk}_{\downarrow}^{QT}(v)$ denote the subcomplexes of $\mathrm{lk}_{\downarrow}^{QV}(v)$ spanned by the vertex sets
\[ \{ (m, m+1, p) \mid m \in \{ 1, \ldots, k-1 \}; p \in \{1, \ldots, \ell \} \}, \]
and
\[ \{ (m, m+1, p) \mid m \in \{ 1, \ldots, k-1 \}; p \in \{1, \ldots, \ell \} \} \, \bigcup \, \{ (k, 1, p) \mid
p \in \{1, \ldots, \ell \} \},\]
respectively.
\end{definition}

\begin{proposition}
Let $v = x^{k}a^{\ell}$. 
The complexes $\mathrm{lk}_{\downarrow}^{QF}(v)$ and $\mathrm{lk}_{\downarrow}^{QT}(v)$ are isomorphic to the descending links of the word $v$ in the subcomplexes $D(\mathcal{P},v)$ and
$D_{a}(\mathcal{P},v)$ of $D_{b}(\mathcal{P},v)$, respectively.
\end{proposition}

\begin{proof}
This follows from the description of $\mathrm{lk}(v)$ from Proposition \ref{proposition:descriptionoflink}, and from the description of the embedding of the complexes for $QF$ and $QT$ into the complex for $QV$ (Proposition \ref{proposition:CAT0cubical}).
\end{proof}

\begin{remark}
As defined in Definition \ref{definition:QFQTlinks}, the vertices of $\mathrm{lk}^{QF}(v)$ and $\mathrm{lk}^{QT}(v)$ are $3$-tuples of integers. We will sometimes think of the final coordinate as a ``color" coordinate, since it has greater freedom of movement than the first two coordinates (which are the ``$x$''-coordinates, and therefore constrained by the requirements that the associated diagrams be planar or annular). We will especially use the color language in the proof of Proposition \ref{proposition:linkQFQT}.
\end{remark}

\subsubsection{The connectivity of the descending links in $K_{QF}$ and $K_{QT}$} \label{subsubsection:links}

We will determine the connectivity of the links $\mathrm{lk}_{\downarrow}^{QF}(v)$ and 
$\mathrm{lk}_{\downarrow}^{QT}(v)$ (for $v = x^{k}a^{\ell}$) with the aid of covers by subcomplexes. The Nerve theorem will be an important tool for us.

\begin{theorem} \cite{Bjorner}
(The Nerve Theorem) Let $\Delta$ be a simplicial
complex and $(\Delta_{i})_{i \in I}$ a family of
subcomplexes such that $\Delta = \bigcup_{i \in I} \Delta_{i}$. If every non-empty finite intersection
$\Delta_{i_{1}} \cap \ldots \cap \Delta_{i_{t}}$
is $(k-t+1)$-connected, then $\Delta$ is $k$-connected if and only if the nerve $\mathcal{N}(\Delta_{i})$ is $k$-connected. \qed
\end{theorem}

\begin{proposition} \label{proposition:linkQFQT}
The complexes $\mathrm{lk}_{\downarrow}^{QF}(x^{k}a^{\ell})$ and $\mathrm{lk}_{\downarrow}^{QT}(x^{k}a^{\ell})$ are 
$n$-connected ($n \geq 0$) if $k \geq 3n+5$ and $\ell \geq 2n+3$.
\end{proposition}

\begin{proof}
The proofs are nearly identical, whether we consider $\mathrm{lk}_{\downarrow}^{QF}(x^{k}a^{\ell})$ or $\mathrm{lk}_{\downarrow}^{QT}(x^{k}a^{\ell})$; we will give a detailed proof in the former case.
 
We prove the statement by induction on $n$. Note that 
the given link is non-empty when $k \geq 2$ and $\ell \geq 1$. 

Consider the base case $n=0$. The sequence
\[ (1,2,\alpha), (3,4,\beta), (1,2,\gamma) \]
determines an edge-path in $\mathrm{lk}_{\downarrow}^{QF}(x^{k}a^{\ell})$, where $\alpha$, $\beta$, $\gamma \in \{ 1, \ldots, \ell \}$
are all distinct. It follows that any two vertices of the form 
$(1,2,\alpha)$ can be connected by an edge-path. Clearly, any vertex of
the form $(m,m+1,\beta)$ $(m > 2)$ is adjacent to some $(1,2,\alpha)$. Finally, note that any vertex $(2,3,\beta)$ can be connected to a vertex of the form $(1,2,\alpha)$ by the following edge-path:
\[ (2,3,\beta), (4,5,\gamma), (1,2,\alpha). \]
It now follows easily that $\mathrm{lk}_{\downarrow}^{QF}(x^{k}a^{\ell})$ is connected if $k \geq 5$ and $\ell \geq 3$.

Now assume that the statement of the proposition is true for $0 \leq j \leq n$. We consider 
$\mathrm{lk}_{\downarrow}^{QF}(x^{k}a^{\ell})$, where $k \geq 3(n+1) + 5$ and $\ell \geq 2(n+1) + 3$. Following the basic strategy of \cite{Far(th)} (in the proof of Proposition 4.11), we would like to cover $\mathrm{lk}_{\downarrow}^{QF}(x^{k}a^{\ell})$ by the simplicial neighborhoods of the vertices
\[ \mathcal{C} = \{ (1,2,\alpha), (2,3,\alpha) \mid \alpha \in \{ 1, \ldots, \ell \} \}. \]
Unfortunately, these simplicial neighborhoods do not cover 
$\mathrm{lk}_{\downarrow}^{QF}(x^{k}a^{\ell})$ in general. (If $k >> \ell$, then it is possible to find a simplex $S$ that ``uses" all of the colors $1, \ldots, \ell$, but none of the vertices in the above collection; such a simplex clearly cannot be part of the simplicial neighborhood of any of the above vertices.) We will instead consider the $(n+2)$-skeleton of $\mathrm{lk}_{\downarrow}^{QF}(x^{k}a^{\ell})$. We claim that the $(n+2)$-skeleton is indeed 
covered by the $(n+2)$-skeletons of the simplicial neighborhoods considered above, as we now prove.

Let $S = \{ (m_{1}, m_{1}+1, \alpha_{1}), \ldots,
(m_{p}, m_{p}+ 1, \alpha_{p} ) \}$ be a simplex in the $(n+2)$-skeleton. If one of the $m_{i}$ is $1$ or $2$, then it is clear that the given simplex is in the simplicial neighborhood of some vertex in $\mathcal{C}$, and the desired conclusion follows. If all of the $m_{i}$ are greater than $2$, then, since $p \leq n+3 < 2n+5 \leq \ell$, there is a color $\beta \in (\{ 1, \ldots, \ell \} - \{ \alpha_{1}, \ldots, \alpha_{p} \})$. The simplex $S$ is therefore in the simplicial neighborhood of
$(1,2,\beta)$. Thus clearly $S$ is in the $(n+2)$-skeleton of the simplicial neighborhood of $(1,2,\beta)$. This proves the claim.

We let $C(m,m+1,\beta)$ denote the $(n+2)$-skeleton of the simplicial neighborhood of $(m,m+1, \beta)$ in
$\mathrm{lk}_{\downarrow}^{QF}(x^{k}a^{\ell})$. We consider the
cover
\[ \widehat{\mathcal{C}} = \{ C(m,m+1, \beta) \mid m \in \{1, 2 \}; \beta \in \{ 1, \ldots, \ell \} \}. \]

We first claim that the nerve 
$\mathcal{N}( \widehat{\mathcal{C}})$ is $(n+1)$-connected. 
We will prove this by showing that every $(n+3)$ vertices of the nerve span an $(n+2)$-simplex of the nerve. It will then follow that the nerve contains the $(n+2)$-skeleton of a high-dimensional simplex, and is therefore $(n+1)$-connected. So, take any $n+3$ members of $\widehat{\mathcal{C}}$. These elements
\[ C(m_{1}, m_{1}+1, \beta_{1}), \ldots, C(m_{n+3}, m_{n+3} + 1, \beta_{n+3}) \]
$(m_{i} \in \{ 1, 2 \})$ collectively ``use" at most $n+3$ colors, out of a total of at least $2n+5$. Thus, there is a $\beta \in (\{ 1, \ldots, \ell \} - \{\beta_{1}, \ldots, \beta_{n+3} \})$, and it follows that the intersection
\[ \bigcap_{i=1}^{n+3} C(m_{i}, m_{i+1}, \beta_{i}) \]
is non-empty. (It contains the vertex $(4,5,\beta)$, for instance.) This proves that the nerve $\mathcal{N}(\widehat{\mathcal{C}})$ is $(n+1)$-connected.

Now we want to prove that any non-empty $t$-fold intersection of
members of $\widehat{\mathcal{C}}$ is $(n-t+2)$-connected. Thus, we consider the intersection
\[ \bigcap_{i=1}^{t} C(m_{i}, m_{i}+1, \beta_{i}), \]
where $m_{i} \in \{ 1, 2 \}$. One easily checks that this intersection consists of all $q$-simplices $(q \leq n+2)$ on the set
\[ \{ (m, m+1, \gamma) \mid m \in \{ 3, \ldots, k-1 \}; \gamma \in (\{ 1, \ldots, \ell \} - \{ \beta_{1}, \ldots, \beta_{i} \}) 
\} \]
or on the set
\[ \{ (m, m+1, \gamma) \mid m \in \{ 4, \ldots, k-1 \}; \gamma \in (\{ 1, \ldots, \ell \} - \{ \beta_{1}, \ldots, \beta_{i} \}) 
\}. \]
It follows directly that the intersection may be identified with 
the $(n+2)$-skeleton of $\mathrm{lk}_{\downarrow}^{QF}(x^{k'}a^{\ell'})$, where $k' \geq k - 3$ and $\ell' \geq \ell - t$. We note first that 
\[ k' \geq k - 3 \geq 3(n+1) + 5 - 3 = 3n+5 \geq 3(n-t+2) + 5 \]
and
\[ \ell' \geq \ell - t \geq 2n+5-t. \]
We want to prove that $2n+5-t \geq 2(n-t+2)+3$, but this is easily seen to be equivalent to the inequality $t \geq 2$. 

We can now conclude that the intersection in question is the same as the $(n+2)$-skeleton 
of $\mathrm{lk}_{\downarrow}^{QF}(x^{k'}a^{\ell'})$, where $k' \geq 3(n-t+2) + 5$
and $\ell' \geq 2(n-t+2)+3$. It follows directly that the intersection is $(n-t+2)$-connected,
as required.

It follows from the Nerve theorem that $\mathrm{lk}_{\downarrow}^{QF}(x^{k}a^{\ell})$ is $(n+1)$-connected for $k \geq 3(n+1) + 5$ and $\ell \geq 2(n+1) + 3$, as required. This completes the induction, and proves the proposition in the case of $QF$.

Note that the required connectivity of the complex $\mathrm{lk}_{\downarrow}^{QT}(x^{k}a^{\ell})$ is proved almost exactly as above. In fact, the intersection of the cones $C(m,m+1, \beta)$ from the last part of the proof is isomorphic to a descending link $\mathrm{lk}_{\downarrow}^{QF}(x^{k'}a^{\ell'})$ (rather than 
a descending link of the form $\mathrm{lk}_{\downarrow}^{QT}(x^{k'}a^{\ell'})$, as one might expect).
\end{proof}

\subsubsection{The connectivity of the descending links in $K_{QV}$} \label{subsubsection:QVlinks}
Here we will determine the connectivity of the descending links $\mathrm{lk}_{\downarrow}^{QV}(x^{k}a^{\ell})$
by a method analogous to the one from \ref{subsubsection:links}. 

The argument makes use of the following lemma (Lemma 6 from \cite{Far(fp)}).

\begin{lemma} \label{farleylemma}
Let $K$ be a finite flag complex, and let $n \geq 0$.
The complex $K$ is $n$-connected if, for any collection of vertices $S \subseteq K^{(0)}$, $|S| \geq 2$, the intersection 
$\bigcap_{v \in S} \mathrm{lk}(v)$ is $(n+1-|S|)$-connected.
\qed
\end{lemma}

\begin{proposition} \label{linkconnectednessV}
For $n \geq 0$, the descending link \lkb{k}{\ell} is $n$-connected when $k \geq 4n+5$ and $l \geq 2n+3$.
\end{proposition}

\begin{proof}
We will use induction to show that this is true for any $n \geq 0$.

Consider $\lkb{k}{\ell}$; let $k \geq 5$ and $\ell \geq 3$. We will show that $\lkb{k}{\ell}$ is $0$-connected.
Let $(m,n,\alpha)$ and $(m',n', \alpha')$ be arbitrary vertices of $\lkb{k}{\ell}$. Note that, if $\{ m, n \} \cup 
\{ m', n'\}$ has three or fewer elements, then the vertices in question are clearly connected by a path, since
$(m'', n'', \alpha'')$ is adjacent to both vertices if
$m'', n'' \notin \{ m, n \} \cup \{ m', n' \}$
and $\alpha'' \notin \{ \alpha', \alpha'' \}$. Thus, we may assume that $m,n, m',$ and $n'$ are all different. In this case, there is nothing to prove unless $\alpha = \alpha'$ (otherwise the vertices are adjacent by definition). Let 
$M \notin \{m, n, m', n' \}$. We note that $(m', M, \beta)$
($\beta \neq \alpha$) is adjacent to $(m,n,\alpha)$. Now, since $(m', M, \beta)$ may be connected to $(m',n', \alpha')$ by the previous argument, it follows that $(m,n,\alpha)$
and $(m',n',\alpha')$ may be connected by a path. This proves that $\lkb{k}{\ell}$ is $0$-connected.

Now let $k \geq 4(n+1)+5$ and $\ell \geq 2(n+1)+3$, where $n \geq 0$. We assume that the statement of the proposition is true for $0 \leq j \leq n$. Let $S$ be an arbitrary collection of vertices in $\lkb{k}{\ell}$. We may assume that $|S| \leq n+2$ (otherwise there is nothing to prove). We note that
\[ \bigcap_{v \in S} \mathrm{lk}(v) \]
is simply the full subcomplex of $\lkb{k}{\ell}$ on the vertices $\hat{v}$ that have no overlap with any of the $v \in S$. (I.e., $\hat{v}$ and $v$ represent disjoint applications of relations (as in Definition \ref{definition:disjointapp}) to $x^{k}a^{\ell}$, for all $v 
\in S$.) It follows that
\[ \bigcap_{v \in S} \mathrm{lk}(v) \cong \lkb{k'}{\ell'}, \]
where $k' \geq 4n+9 - 2|S|$ and $\ell' \geq 2n+5 - |S|$. By Lemma \ref{farleylemma} and induction, it now suffices to show that 
\[ 4n+9 - 2|S| \geq 4(n+1- |S|) + 5 \]
and
\[ 2n+5 - |S| \geq 2(n+1 - |S|) + 3.\] Both of these inequalities are obvious. This completes the induction.
\end{proof}

\subsection{Conclusion: Proof of the $F_{\infty}$ property}
\label{subsection:conclusion}

It is now straightforward to complete the proof of the $F_{\infty}$ property.

\begin{proposition} \label{proposition:collection}
Let $X = K_{QF}, K_{QT},$ or $K_{QV}$.
\begin{enumerate}
\item If $X = K_{QF}$ or $K_{QT}$, then $X_{k}$ is $n$-connected if $k \geq 3n+5$. 
\item If $X = K_{QV}$, then $X_{k}$ is $n$-connected if
$k \geq 4n+5$.
\end{enumerate}
\end{proposition}

\begin{proof}
If $X = K_{QF}$ or $K_{QT}$, then, by Propositions 
\ref{proposition:standard} and \ref{proposition:linkQFQT}, the descending link
$\mathrm{lk}_{\downarrow}(x^{k}a^{k-1})$ is $n$-connected provided that $k \geq 3n+5$. 

Similarly, if $X = K_{QV}$, then, by Propositions 
\ref{proposition:standard} \ref{linkconnectednessV}, the descending link $\mathrm{lk}_{\downarrow}(x^{k}a^{k-1})$ is $n$-connected provided that $k \geq 4n+5$.
\end{proof}

\begin{theorem} \label{theorem:bigone}
The groups $QF$, $QT$, and $QV$ have type $F_{\infty}$.
\end{theorem}

\begin{proof}
The complexes $X = K_{QF}$, $K_{QT}$, and $K_{QV}$ admit actions by $QF$, $QT$, and $QV$ (respectively) satisfying (1) and (2) from Theorem \ref{theorem:Brown}. This was established by the end of Subsection
\ref{subsection:overview}. In Subsection \ref{subsection:filtration}, we established that the complexes $X$ admit filtrations by subcomplexes
$\{ X_{k} \}$ such that (3) (a)-(c) from Theorem \ref{theorem:Brown} are satisfied. By Proposition \ref{proposition:collection}, these filtrations satisfy (3) (d) from Theorem \ref{theorem:Brown} as well.
It now follows from Theorem \ref{theorem:Brown} that $QF$, $QT$, and $QV$ have type $F_{\infty}$.
\end{proof}

\section{Some generalizations} \label{section:generalization}

Nucinkis and St. John-Green introduced the groups $\bar{Q}T$ and $\bar{Q}V$, and proved that both groups have type $F_{\infty}$. The group $\bar{Q}V$ is defined just as $QV$ was  (see Definition \ref{definition:QV}) , except that $\bar{Q}V$ is a set of automorphisms of $\mathcal{T} \cup \{ \ast \}$, where $\ast$ is an isolated point, rather than a set of automorphisms of $\mathcal{T}$, as in the case of $QV$. There is still a natural action of $\bar{Q}V$ on the 
set of ends of $\mathcal{T}$, which yields a homomorphism $\bar{Q}V \rightarrow V$. The inverse image of Thompson's group $T$ under this homomorphism is $\bar{Q}T$. 

In the vocabulary of the main body of the paper, $\bar{Q}V$ is the braided diagram group
$D_{b}(\mathcal{P}, xa)$, where $\mathcal{P} = \langle x, a \mid x = xax \rangle$. The group 
$\bar{Q}T$ is the subgroup of $\bar{Q}V$ consisting of diagrams that become (or remain) annular when all wires labelled ``$a$'' are erased. The groups $\bar{Q}V$ and $\bar{Q}T$ then act on cubical complexes analogous to the ones described in Subsections \ref{subsection:diagramcomplex} and \ref{subsection:QVdiagramgroup}: $\bar{Q}V$ acts on $\widetilde{K}_{b}(\mathcal{P}, xa)$ and $\bar{Q}T$ acts on the convex subcomplex spanned by diagrams that are annular (in the above extended sense). The proofs of the $F_{\infty}$ property from the main body of the paper apply to these groups with no essential changes. 

More generally, we can define a class of groups as follows. Let $\mathcal{F}_{k,\ell}$ denote an ordered forest consisting of $k$ infinite binary trees and $\ell$ isolated vertices. We let $QV_{k,\ell}$ denote the group of all bijections of the vertices of $\mathcal{F}_{k, \ell}$ that preserve left- and right-children, with at most finitely many exceptions (as in Definition \ref{definition:QV}). We let $QF_{k, \ell}$ and $QT_{k, \ell}$ be the subgroups of $QV_{k, \ell}$ that preserve the linear (respectively, cyclic) ordering of the ends. The group $QV_{k, \ell}$ is simply the braided diagram group $D_{b}(\mathcal{P}, x^{k}a^{\ell})$ ($\mathcal{P}$ as above), and the groups $QF_{k, \ell}$ and $QT_{k, \ell}$ have obvious definitions, analogous to the ones given above.
A minor variant of the main argument now proves
\begin{theorem}
The groups $QF_{k, \ell}$, $QT_{k, \ell}$, and $QV_{k, \ell}$ have type $F_{\infty}$, for $k, \ell \geq 0$.
\qed
\end{theorem}

Further extensions are possible. For instance, we could consider the case of rooted ordered $n$-ary trees. Let $\mathcal{T}^{n}_{k, \ell}$ denote the ordered forest of $k$ infinite $n$-ary trees and
$\ell$ isolated vertices. We can let $QV^{n}_{k, \ell}$ denote the group of bijections $h$ of $(\mathcal{T}^{n}_{k, \ell})^{0}$ that send the $j$th child of a vertex $v$ to the $j$th child of the vertex $h(v)$, with at most finitely many exceptions. This group is isomorphic to the braided diagram group
$D_{b}(\mathcal{P}_{n}, x^{k}a^{\ell})$, where $\mathcal{P}_{n} = \langle x, a \mid x = x^{n}a \rangle$.
This construction also yields its own ``$QF$'' and ``$QT$" versions, in a natural way. All such groups will have type $F_{\infty}$, provided that $n \geq 2$. The argument differs from the main argument of this paper in only minor ways (for instance, the bounds in the analogs of Propositions
\ref{proposition:linkQFQT} and \ref{linkconnectednessV} will be different).  

Even more generally, we could define a tree using a tree-like semigroup presentation $\mathcal{P}' = \langle \Sigma \mid \mathcal{R} \rangle$ (see Definition 4.1 from \cite{FH2}). In such a presentation, each relation has the form $x = x_{1}x_{2} \ldots x_{i}$, and a given $x \in \Sigma$ is the left-hand side of at most one relation in $\mathcal{R}$. For any fixed $x \in \Sigma$, there is a natural simplicial tree
$T_{(\mathcal{P}', x)}$ (as defined in the proof of Theorem 4.12 from \cite{FH2}). We can construct $T_{(\mathcal{P}',x)}$ inductively as follows. We begin with a root labelled $x$. This vertex is adjacent to children labelled (from left to right) $x_{1}$, $x_{2}$, $\ldots$, $x_{i}$ (respectively), if $(x = x_{1}x_{2}\ldots x_{i}) \in \mathcal{R}$. One similarly introduces and labels the children of $x_{1}$, $\ldots$, $x_{i}$ using the relations of the
form $(x_{\alpha} = w_{\alpha}) \in \mathcal{R}$ ($1 \leq \alpha \leq i$; $w_{\alpha}$ is a word in the generators $\Sigma$). The result is easily seen to be a rooted, ordered, labelled simplicial tree. The reader can easily verify that $T_{(\mathcal{P}, x)}$ is the usual infinite binary tree if $\mathcal{P} = \langle x \mid x = x^{2} \rangle$, for instance. Now, assuming that $a \notin \Sigma$, we define $\mathcal{P}'' = \langle \Sigma \cup \{ a \} \mid \mathcal{R}' \rangle$, where $\mathcal{R}'$ is the result of replacing each relation $(x = x_{1} \ldots x_{i}) \in \mathcal{R}$ with the relation $x = x_{1} \ldots x_{i} a$. (Here, the letter ``$a$'' is again being used to represent an isolated vertex.) The braided diagram group $D_{b}(\mathcal{P}'',x)$ should  be isomorphic to the group of quasi-automorphisms of $T_{(\mathcal{P}', x)}$ under appropriate hypotheses on the original tree-like semigroup presentation $\mathcal{P}'$. (For instance, one would need to ensure that the trees $T_{(\mathcal{P}',x)}$ and $T_{(\mathcal{P}',y)}$ are not isomorphic for
different $x$ and $y$; failure to do this would make the braided diagram group strictly smaller than the corresponding quasi-automorphism group.) It is reasonable to expect that such groups will often be of type $F_{\infty}$, but we will not try to guess at the appropriate hypotheses here.

It is worth adding a final cautionary note. We cannot expect all groups in the family sketched above to have type $F_{\infty}$.
The Houghton group $H_{n}$ can be described as the ``$QF$'' group associated to the forest
$\mathcal{T}^{1}_{n,0}$. This group was shown to have type $F_{n-1}$ but not type $F_{n}$ by Ken Brown \cite{Brown}. In fact, it seems that adding even a single rooted $1$-ary tree (i.e., cellulated ray) to a forest of higher-valence rooted trees will destroy the $F_{\infty}$ property.

\bibliographystyle{plain}
\bibliography{mainrev}

\end{document}